\documentclass[12pt]{article}
\usepackage{amssymb,amsmath,amsfonts}
\usepackage{stmaryrd}
\usepackage{color}
\usepackage{enumitem}

\usepackage[small,bf,hang]{caption} 

\usepackage{epsfig}
\usepackage[font=small,skip=4pt]{caption}

\def\BBox{\kern  -0.2cm\hbox{\vrule width 0.2cm height 0.2cm}}

\newtheorem{theorem}{Theorem}[section]

\newtheorem{conjecture}[theorem]{Conjecture}

\graphicspath{{figures/}}

\setlist{noitemsep}

\title{\textbf{A note on 2--bisections of claw--free cubic graphs}}

\author{Mari\'{e}n Abreu$^{1}$\thanks{The research that led to the present paper was partially supported by a grant of the group GNSAGA of INdAM and by the Italian Ministry Research Project PRIN 2012 ``Geometric Structures, Combinatorics and their Applications''.}, Jan Goedgebeur$^2$\thanks{Supported by a Postdoctoral Fellowship of the Research Foundation Flanders (FWO).
\newline {\em Email addresses:} marien.abreu@unibas.it (M. Abreu),~ jan.goedgebeur@ugent.be (J. Goedgebeur), ~ domenico.labbate@unibas.it (D. Labbate), ~ giuseppe.mazzuoccolo@univr.it (G. Mazzuoccolo)}, \\ Domenico Labbate$^{1}$$^*$, Giuseppe Mazzuoccolo$^{3}$$^*$
\\[2ex]
\footnotesize $^1$ Dipartimento di Matematica, Universit\`{a} degli Studi della Basilicata,\\[-3mm]
\footnotesize Viale dell'Ateneo Lucano, I-85100 Potenza, Italy.\\
\footnotesize $^2$ Deparment of Applied Mathematics, Computer Science and Statistics\\[-3mm]
\footnotesize Ghent University, Krijgslaan 281 - S9, 9000 Ghent, Belgium.\\
\footnotesize $^3$ Dipartimento di Informatica,\\[-3mm]
\footnotesize Universit\`{a} degli Studi di Verona, Strada le Grazie 15, 37134 Verona, Italy.}

\date{}

\begin{document}
\maketitle

\begin{abstract}
A \emph{$k$--bisection} of a bridgeless cubic graph $G$ is a $2$--colouring of its vertex set such that the colour classes have the same cardinality and all connected components in the two subgraphs induced by the colour classes have order at most $k$.
Ban and Linial conjectured that {\em every bridgeless cubic graph admits a $2$--bisection except for the Petersen graph}.

In this note, we prove Ban--Linial's conjecture for claw--free cubic graphs.

\end{abstract}



\section{Introduction}

All graphs considered in this paper are finite and simple (without loops or multiple edges). Most of our terminology is standard; for further definitions and notation not explicitly stated in the paper, please refer to~\cite{BM}.

A {\em bisection} of a cubic graph $G=(V,E)$ is a partition of its vertex set $V$ into two disjoint subsets $({\cal B}, {\cal W})$ of the same cardinality. We will often identify a bisection of $G$ with the vertex colouring of $G$ with two colours $B$ (black) and $W$ (white), such that every vertex of ${\cal B}$ and ${\cal W}$ has colour $B$ and $W$, respectively.  Note that the colouring does not need to be proper and in what follows we refer to a connected component of the subgraphs induced by a colour class as a \emph{monochromatic component}.
Following this terminology, we define a \emph{$k$--bisection} of a graph $G$ as a $2$--colouring $c$ of the vertex set $V(G)$ such that:

\begin{enumerate}[label=(\arabic*)]
\item[(i)]
$|{\cal B}|=|{\cal W}|$ (i.e.\ it is a bisection), and
\item[(ii)]
each monochromatic component has at most $k$ vertices.
\end{enumerate}


There are several papers in literature considering $2$--colourings of regular graphs which satisfy condition $(ii)$, but not necessarily condition $(i)$, see \cite{ADOV, BS, HST, LMST}. In particular, it is easy to see that every cubic graph has a $2$--colouring where all monochromatic connected components are of order at most $2$. But, in general, such a colouring does not satisfy condition $(i)$, so it is not a $2$--bisection. Thus, the existence of a $2$--bisection in a cubic graph is not guaranteed. For instance, the Petersen graph does not admit a $2$--bisection. However, the Petersen graph is an exception since it is the only known bridgeless cubic graph without a $2$--bisection. This led Ban and Linial to pose the following conjecture:

\begin{conjecture}[Ban--Linial~\cite{BL16}]\label{banlinial}
Every bridgeless cubic graph admits a $2$--bisection, except for the Petersen graph.
\end{conjecture}

As far as we know, the best positive result in this direction is given in \cite{EMT15}, where it is proven that every cubic graph (not necessarily bridgeless) admits a $3$-bisection.

A main difficulty in approaching this conjecture is the presence of both local and global conditions, that is: conditions (ii) and (i) in the definition of $k$-bisection, respectively. Some standard reductions on the girth and the connectivity of a minimal possible counterexample (which can be easily deduced in similar problems) do not work in this context. In particular, so far, we cannot exclude the presence of a $3$-cycle or a $2$-edge-cut in a minimal counterexample.

Recently, we have given a detailed insight into Ban--Linial's Conjecture (as well as some other related conjectures) in \cite{AGLM}, where we also presented theoretical and computational evidence for it. In particular, Ban--Linial's Conjecture is proven for bridgeless cubic graphs in the special case of cycle permutation graphs.

In this short note, we provide further evidence for it by proving Ban--Linial's Conjecture for claw--free cubic graphs (cf.\ Theorem \ref{clfr2bisec}).

\section{$2$-bisections of claw-free cubic graphs}\label{clfree}

The complete bipartite graph $K_{1,3}$ with colour classes of size $1$ and $3$, respectively, is usually called the {\em claw graph}.
A {\em claw--free graph} is a graph that does not have a claw as an induced subgraph.

In~\cite{CS05}, general claw-free graphs are characterized. In~\cite{Oum11}, Oum has characterized simple, claw-free bridgeless cubic graphs (cf.\ Theorem \ref{clfrstru}). In order to present such a characterization we need some preliminary definitions.

Let $G$ be a bridgeless claw-free cubic graph. An induced subgraph of $G$ that is isomorphic to $K_4-e$ is called a {\it diamond} and denoted by $D$ in what follows.
A {\it string of diamonds} of $G$ is a maximal sequence $D_{1},...,D_{k}$ of diamonds, in which $D_{i}$ has a vertex adjacent to a vertex of $D_{i+1}$, $1\leq i \leq k-1$ (see Figure~\ref{sod}). Moreover, if $G$ is a connected claw-free simple cubic graph such that each vertex lies in a diamond, then $G$ is called a {\it ring of diamonds} (see Figure~\ref{rod}).

\begin{figure}[h]
	\centering
		\includegraphics[height=2.5cm]{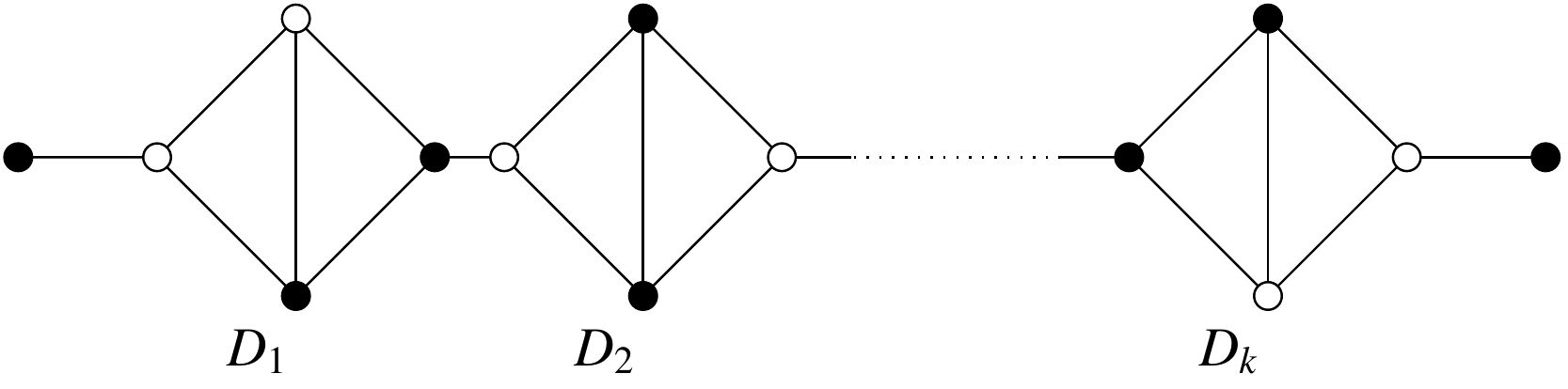}
		\caption{A string of diamonds}\label{sod}		
\end{figure}

\begin{figure}[h]
	\centering
		\includegraphics[height=6cm]{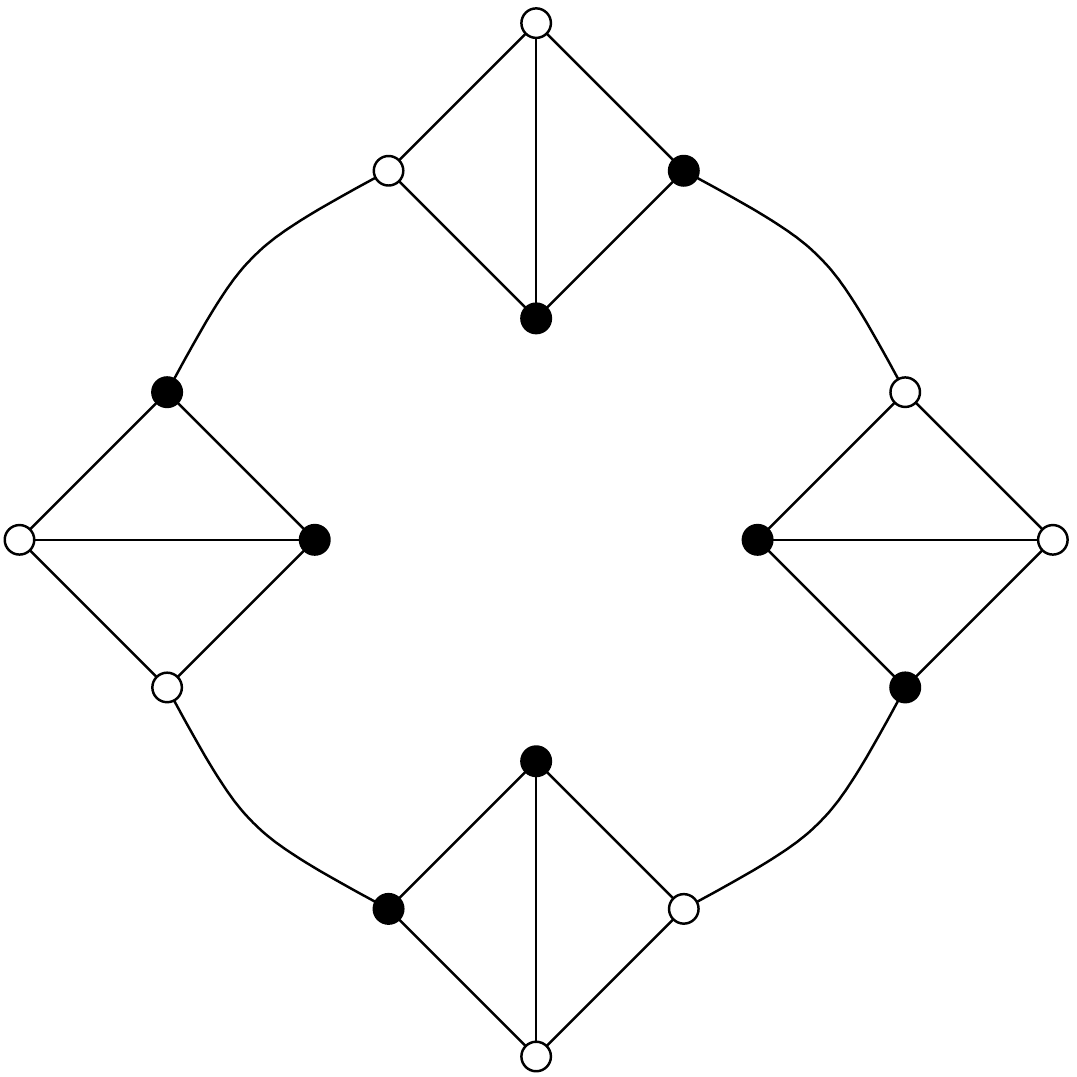}
		\caption{A ring of diamonds}\label{rod}		
\end{figure}

\begin{theorem}[Oum \cite{Oum11}]\label{clfrstru}
A graph $G$ is bridgeless claw-free cubic if and only if either

\begin{enumerate}[label=(\roman*)]
\item $G \simeq K_4$
\item $G$ is a ring of diamonds
\item $G$ can be built from a bridgeless cubic multigraph $H$ by replacing some edges of $H$ with strings of diamonds and replacing each vertex of $H$ with a triangle.
\end{enumerate}
\end{theorem}

The main purpose of this section is proving Conjecture \ref{banlinial} for the class of claw-free cubic graphs (cf.\ Theorem \ref{clfr2bisec}).

We recall that a $2$--bisection is a bisection $(B,W)$ such that the connected components in the two induced subgraphs $G[B]$ and $G[W]$ have order at most two. In other words, every induced subgraph is a union of isolated vertices and isolated edges.

Consider a $2$-bisection of a claw-free cubic graph and consider the colouring of vertices of a diamond of $G$. Since any three vertices of a diamond induce a connected subgraph, we have exactly two vertices of the diamond in each colour class. Then, we have only two possibilities: either the two non-adjacent vertices receive different colours (see the diamond $D_1$ in Figure~\ref{sod}) or the same colour (see the diamond $D_2$ in Figure~\ref{sod}). We call the former colouring {\em asymmetric} the latter one {\em symmetric}.

\begin{theorem}\label{clfr2bisec}
Let $G$ be a bridgeless claw-free cubic graph. Then, $G$ admits a $2$-bisection.
\end{theorem}

\begin{proof}
It is straightforward that $K_4$ and every ring of diamonds admit a $2$-bisection: we can trivially generalize the example from Figure~\ref{rod} by colouring each diamond of the ring with a suitable chosen asymmetric colouring.

Consider $G$ as in Theorem~\ref{clfrstru}$(iii)$: let $H$ be the underlying bridgeless cubic multigraph and let $H^{\bigtriangleup}$ be the graph obtained by replacing each vertex of $H$ with a triangle. Then $G$ is built from $H^{\bigtriangleup}$ by replacing some edges, not in the new triangles, with strings of diamonds.

First of all, we prove that if $H^{\bigtriangleup}$ admits a $2$-bisection, then $G$ admits a $2$-bisection.
We start from a $2$-bisection of $H^{\bigtriangleup}$ and we colour all vertices of the strings of diamonds of $G$ such that we obtain a $2$-bisection of $G$.
Let $uv$ be an edge of $H^{\bigtriangleup}$ which is replaced with a string of diamonds in $G$.
If $u$ and $v$ receive the same colour, say white, in a $2$-bisection of $H^{\bigtriangleup}$, then we colour the vertices of $G$ corresponding to $u$ and $v$ with the colour white (the same colour they have in $H^{\bigtriangleup}$), and we colour the vertices of each diamond of the string with a symmetric $2$-colouring such that non-adjacent vertices are black.
On the other hand, if $u$ and $v$ receive different colours in the $2$-bisection of $H^{\bigtriangleup}$, then, again, we colour the vertices of $G$ corresponding to $u$ and $v$ with the same colour they have in $H^{\bigtriangleup}$, and we colour the vertices of each diamond of the string with an asymmetric $2$-colouring by paying attention that adjacent vertices of distinct diamonds receive different colours.

Now we prove that $H^{\bigtriangleup}$ admits a $2$-bisection. Let $M$ be a perfect matching of $H$ (since $H$ is bridgeless, the existence of $M$ is guaranteed by a classical theorem of Petersen, see \cite{Pet}). The complement of $M$ in $H$ is a $2$--factor of $H$ (it could have circuits of length $2$, since $H$ is a multigraph), say $C_1, \ldots C_t$ are its circuits. Every edge of $M$ naturally corresponds to an edge of $H^{\bigtriangleup}$: denote by $M'$ the set of those edges in $H^{\bigtriangleup}$. Every connected component of the complement of $M'$ in $H^{\bigtriangleup}$ is isomorphic to a circuit $C_i$ having each vertex replaced by a triangle. Denote by $C'_i$ the circuit corresponding to $C_i$ of length $2|C_i|$ (cf. Figure \ref{evencircuit}).

\begin{figure}[h]
	\centering
		\includegraphics[height=4cm]{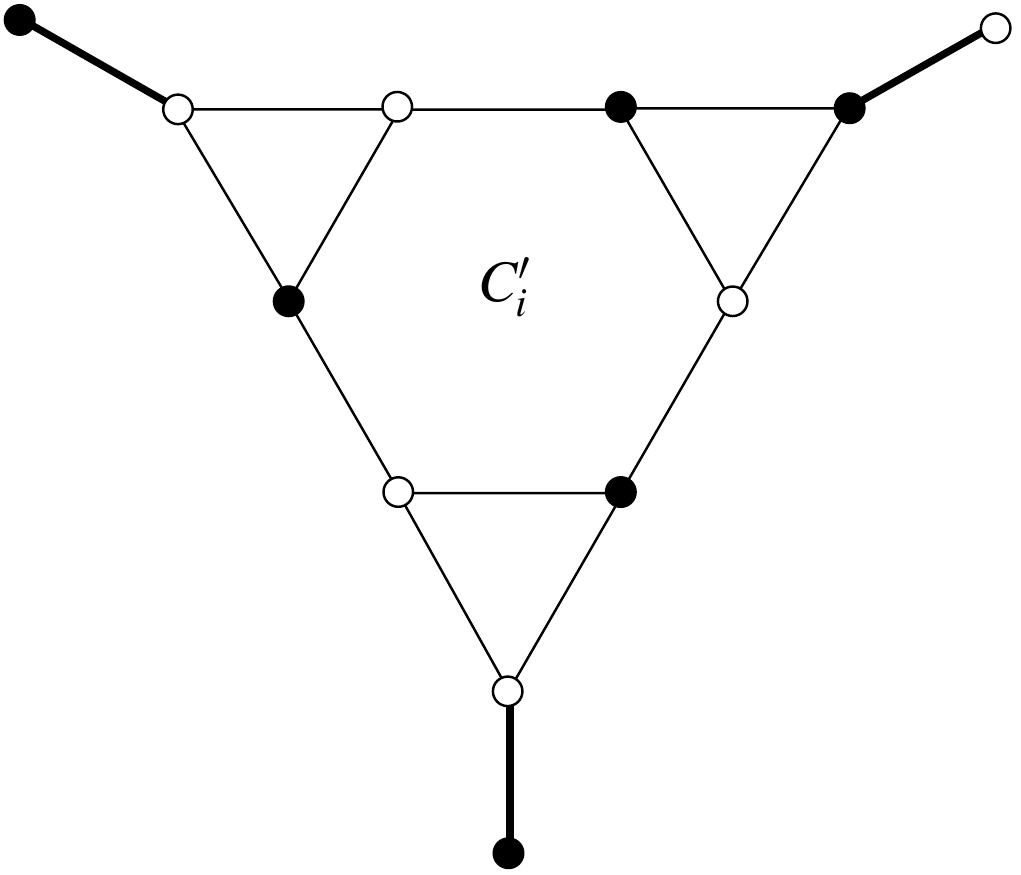}
		\caption{Pictures of $C'_i$ in $H^{\bigtriangleup}$. Edges of $M'$ in bold.}\label{evencircuit}		
\end{figure}

We colour the vertices of the even circuit $C'_i$ alternately black and white. Moreover, we colour the ends of every edge of $M'$ with different colours. It is straightforward that such a $2$-colouring is a bisection. Furthermore, every vertex of a circuit $C'_i$ is adjacent to at least two vertices of the opposite colour and every vertex not in a circuit $C_i'$ is adjacent to exactly two vertices of the opposite colour (its neighbour in $M'$ and one of its further neighbours).
\end{proof}

\end{document}